\documentclass[12pt]{amsart}
\usepackage{latexsym,fancyhdr,amssymb,color,amsmath,amsthm,graphicx,listings,comment}
       \newtheorem{propo}{Proposition}
\newtheorem{lemma}{Lemma}       
\usepackage[section]{placeins}  
\let\paragraph\subsection
\title{On index expectation curvature for manifolds} 
\author{Oliver Knill} 
\date{1/19/2020}
\address{Department of Mathematics \\ Harvard University \\ Cambridge, MA, 02138 }

\begin{document}
\begin{abstract}
Index expectation curvature $K(x) = E[i_f(x)]$  on a compact Riemannian 
$2d$-manifold $M$ is an	expectation of Poincar\'e-Hopf indices $i_f(x)$ and so
satisfies the Gauss-Bonnet relation $\int_M K(x) \; dV(x)=\chi(M)$.
Unlike the Gauss-Bonnet-Chern integrand, these curvatures are in general non-local.
We show that for small $2d$-manifolds $M$ with boundary embedded in a parallelizable 
$2d$-manifold $N$ of definite sectional curvature sign $e$, an index expectation $K(x)$ 
with definite sign $e^d$ exists. The function $K(x)$ is constructed as a product
$\prod_{k} K_k(x)$ of sectional index expectation curvature averages $E[i_k(x)]$ of
a probability space of Morse functions $f$ for which $i_f(x) = \prod i_k(x)$, where
the $i_k$ are independent and so uncorrelated. 
\end{abstract}  
\maketitle

\section{Introduction}

\paragraph{}
The Hopf sign conjecture \cite{Hopf1932,BergerPanorama,BishopGoldberg} states that
{\it a compact Riemannian $2d$-manifold $M$ with definite sectional 
curvature sign $e$ has Euler characteristic $\chi(M)$ with sign $e^d$. }
The problem was stated first explicitly in talks of Hopf given in 
Switzerland and Germany in 1931 \cite{Hopf1932}. 
In the case of non-negative or non-positive curvature, it appeared in 
\cite{BishopGoldberg} and a talk of 1960 \cite{Chern1966}. They appear as problems 8) and 10)
in Yau's list of problems \cite{YauSeminar1982}. As mentioned in 
\cite{BergerPanorama}, earlier work of Hopf already was motivated by these line of questions. 
The answer is settled positively in the case $d=1$ (Gauss-Bonnet) and $d=2$ (Milnor),
if $M$ is a hypersurface in $\mathbb{R}^{2d+1}$ (Hopf) \cite{HopfCurvaturaIntegra}
or in $\mathbb{R}^{2d+2}$ (Weinstein) \cite{Weinstein70}
or in the case of positively sufficiently pinched curvature, because of
sphere theorems (Rauch,Klingenberg,Berger) \cite{Rauch51,GromollKlingenbergMeyer}
spheres and projective spaces in even dimension have positive Euler characteristic. 

\paragraph{}
The question has motivated other developments like comparison geometry \cite{GrovePetersen},
estimating quantities like the radius of injectivity, diameter or 
or diameter sphere theorems (e.g. \cite{AbreschMeyer}).
The Euler-Poincar\'e formula $\chi(M) = \sum_{k=0}^{2d} (-1)^k b_k(M)$ leads to the quest to estimate 
Betti numbers. In the positive curvature case, analytic techniques for harmonic functions
(Bochner) or global bounds on the Betti numbers (Gromov). 

\paragraph{}
The development of higher dimensional Gauss-Bonnet theorems started with Hopf
in 1927 and culminated in the intrinsic 1944 proof of Chern 
\cite{HopfCurvaturaIntegra,Allendoerfer,Fenchel,AllendoerferWeil,Chern44}.
An algebraic Hopf conjecture asked whether the
Gauss-Bonnet-Chern integrand is positive \cite{Weinstein71}. This has 
been confirmed for $d=2$ \cite{Chern1955}, but failed for $d \geq 3$ \cite{Geroch,Klembeck}. 
Other directions of work consisted of constructing examples (see \cite{Ziller,Ziller2,EscherZiller}
for overviews) or establishing the conjecture under some additional symmetry assumptions as
all known positive curvature examples admit a Killing field, an isometric circle action. 

\paragraph{}
We use an integral geometric set-up with the goal to construct a curvature function 
$K$ on $M$ that has the sign $e^d$. The approach revives the early algebraic approach on the problem. 
It is a fresh take on the story however because the so constructed curvature functions $K$ 
are {\bf non-local}. They are a product of averages $K_n(x)$ of sectional curvatures measured 
at various other places in $M$. The ultimate aim is to show the following:

{\it {\bf Revival of the Algebraic Hopf conjecture}: 
a compact Riemannian $2d$-manifold $M$ with definite sectional
curvature sign $e$ has piecewise smooth curvature functions $K_k(x)$ of sign $e$
such that $K(x) = \prod_{k=1}^d K_k(x)$ satisfies $\int_M K(x) \; dV(x)=\chi(M)$ 
and is absolutely continuous. 
\label{revival} 
}

\paragraph{}
For general manifold, the construction can hardly be done globally in a smooth way
if we want a product representation $K(x) = \prod_{k=1}^d K_k(x)$. The manifold needs to be partitioned into polyhedra
with boundary. The product representation then applies in the interior and allows there to control the sign.
Things work for manifolds $M$ with boundary embedded in some other manifold $N$ for which a non-trivial
smooth section of the orthonormal frame bundle $O(N)$ exists. In the same way as with Gauss-Bonnet-Chern
on manifolds with boundary, there is there a boundary contribution which
leads to a contribution $\int_{\delta M} d\kappa(x)$ which is independent of the 
chosen frame bundle. Spaces producing the curvatures $K_n$ are the ones coming from linear functions in an 
ambient space $E$. Having the mass concentrated on finitely many points like if
cut linearly using hyperplanes in the ambient space $E$ into smaller pieces, this could settle
the above algebraic Hopf revival and so the Hopf conjecture. For now, we focus to product manifolds 
or a local situation with small manifolds with boundary. 

\paragraph{}
We start to construct integral geometric curvature functions $K(x)$ in examples. It involves sectional 
index expectation curvature functions $K_k(x)$ which are expectations
of sectional curvature averages $K_k(x)$ integrated over planes parallel through the $k$'th coordinate plane in a
given coordinate system. Constructing $K$ globally for a general compact Riemannian manifold while keeping the product formula
and $\int_M K(x) dV(x)=\chi(M)$ both valid appears already not possible if $M$ is not parallelizable. The reason is that
the construction of $K$ in the interior of $M$ depends on the global coordinate frame bundle. 
And global sections of the orthonormal frame bundle $O(M)$ do not exist in the cases of interest
because a trivial frame bundle would implies $\chi(M)=0$ by Poincar\'e-Hopf, while the aim is to shed light 
where the sign of $\chi(M) e^d$ is non-zero. 

\paragraph{}
The construction works in the product case $M=M_1 \times \cdots \times M_d$, 
where $(M_k,g_k)$ are Riemannian 2-manifolds and
the metric is the product metric. The second case is local, when $M \subset N$ is a 
compact manifold with boundary, where $N=\mathbb{R}^{2d}$ is equipped with a metric with definite curvature sign $e$
and $N$ is parallelizable, like if $N$ is a small open ball of an other Riemannian manifold of sign $e$. 
We can construct a curvature function $K(x)$ that have sign $e^d$ in the interior but
also get curvature on the boundary. 
In the product case, we have sections of $M$, where the curvature is zero but we want to illustrate
how to get the curvature function $K$ that is positive. The construction is local in the sense that
we have to do it in sufficiently small patches $N$ of a given manifold. But the functions $K(x)$ are
{\bf non-local} in the sense that the functions depend on how things have been chopped up. 

\paragraph{}
The reason why the two cases $M=M_1 \times \cdots \times M_d$ and 
$M \subset N$ where $N$ is an open parallelizable manifold, work is that in both cases
there is a global splitting of the tangent bundle into Stiefel sub-bundles of two $2$-dimensional frames.
The product case shows that it is possible to get global smooth functions $K$. 
because it is possible to write down 2-planes parallel to a given 2-plane in $M$.
The second case was pointed already in \cite{DiscreteHopf} but we formulate it here
in the situation, where $M$ can have a boundary and so rather arbitrary Euler characteristic. 

\paragraph{}
With piecewise linear boundary parts, we will be able to 
force the boundary curvature $d\kappa$ to be located on the vertices. This then allows to
glue for a general manifold $M$ (for example using a triangulation). One of the simplest examples 
is the spherical triangle $M \subset N$ where $N$ is an open hemisphere embedded in 
$E=\mathbb{R}^3$. The curvature $K(x)$ in the interior is constant $1/(4\pi)$, 
the curvature on the boundary is located on the vertices. 
Gauss-Bonnet is Harriot's formula and the probabilities can be read off as areas of the $8$ regions 
in which the lines divides the sphere $S$ which with normalized area measure $\mu$ serves as the 
probability space $(\Omega,\mathcal{A},\mu)$ at hand.  See Figure~(\ref{2}).

\section{Index expectation}

\paragraph{}
To review {\bf curvature expectation} \cite{DiscreteHopf} take a compact Riemannian manifold $M$
and look at $\Omega$, the set of Morse functions on $M$. We first assume $M$ to have no
boundary and that we have normalized the volume measure $dV$ on $M$ to be a probability 
measure. The case with boundary requires to look at Morse functions of an other manifold $N$
of the same dimension into which $M$ is embedded. The definition of the index at the boundary 
is then $i_f(x) = 1-\chi(S_r(x) \cap M)$, where $S_r(x)$ is a sphere of sufficiently small 
radius $r$ in $N$. For every $f \in \Omega$, the Poincar\'e-Hopf theorem $\sum_x i_f(x) = \chi(M)$	
expresses the Euler characteristic $\chi(M)$ of $M$ as a sum 
of Poincar\'e-Hopf indices $i_f(x)$ defined by the finite set of critical points $x$ of $f$. 

\paragraph{}
Any probability space $(\Omega,\mathcal{A},\mu)$ of Morse functions 
defines an {\bf index expectation curvature function} 
$$  K(x) = E[i_f(x)]  \; , $$ 
where $E[\cdot]$ denotes the expectation with respect to the probability space. 
This leads to Gauss-Bonnet. Let us give a formal proof.
We first remind that we always assume the volume measure $dV$ on $M$ to be normalized to be 
a probability measure. This is an assumption avoids to mention sphere volume constants. 

\begin{lemma}
Index expectation curvatures satisfy Gauss-Bonnet:
$\int_M K(x) \; dV(x) = \chi(M)$.
\end{lemma} 

\paragraph{}
\begin{proof} 
We use the Fubini theorem to get
\begin{eqnarray*}
 \int_M K(x) \; dV(x) &=& \int_M \int_{\Omega} i_f(x) d\mu(f) dV(x) = \int_{\Omega} \int_M i_f(x) dV(x) d\mu(f) \\
                      &=& \int_{\Omega} \chi(M) \; d\mu(f) = \chi(M) \; .
\end{eqnarray*}
This relation uses generalized functions $i_f(x)$ which are discrete pure point measures on
$M$ supported on finitely many points (due to compactness and the Morse assumption). 
In order to use the classical Fubini theorem, we can mollify the functions $i_{f,\epsilon}(x)$ by
applying a suitable convolution with a mollifier $\phi$ of support in an $\epsilon$ neighborhood of $x$,
run the above proof where $i_{f}(x)$ is replaced by $i_{f,\epsilon}(x)$, 
then take the limit $\epsilon \to 0$. 
\end{proof} 

\paragraph{}
Let $M$ be a smooth manifold with piecewise smooth boundary which is part of an open manifold $N$ without boundary and
the same dimension. Assume then that $N$ is embedded in an ambient Euclidean
space $E$ we can take the probability space of linear functions $f_a(x) = a \cdot x$ with $a \in S$,
where $S$ is the sphere centered at $0$ of volume $1$ in $E$. Take the rotational invariant probability measure
$\mu$ on $S$ and define the probability space $(\Omega = \{ f_a,  a \in S \},\mathcal{A},\mu)$, where $\mathcal{A}$ is 
the Euclidean Borel $\sigma$-algebra defined by the natural distance topology on $S$. Now, 
index expectation defines an absolutely continuous curvature function $K(x)$ in the interior or $M$ and 
a curvature measure $d\kappa(x)$ on the boundary $\delta M$. This produces a Gauss-Bonnet relation with boundary. 

\begin{lemma}
$\chi(M) = \int_{{\rm int}(M)} K(x) \; dV(x) + \int_{\delta M} d\kappa(x)$. 
\end{lemma} 

\begin{proof}
While in the interior, the Poincar\'e-Hopf index $i_f(x) = (-1)^{m(x)}$ is defined through the Morse index $m(x)$
which is the number of negative eigenvalues of the Hessian matrix $H_f(x)$ at a critical point $x$ of $f$, 
we have to clarify what we mean with the index at the boundary. By the piecewise smoothness 
assumption for $x \in \delta N$, we have for a small enough radius $r$ a $d-1$ manifold $S_r(x) \cap M$ with 
boundary in the geodesic sphere $S_r(x)$ of radius $r$ in $N$. The index is defined as $i_f(x) = 1-\chi(S_r(x) \cap M)$. 
This is by the way the same definition which applies in the interior and which applies also in graph theory. 
Now, the Poincar\'e-Hopf theorem applies in this case (see for example \cite{Pugh1968}). Already the case when 
$M$ is a flat triangle in $E=\mathbb{R}^2$ shows that the measure $d\kappa$ is in general not absolutely continuous
(it is there a pure point measure with weight $\alpha/\pi,\beta/\pi,\gamma/\pi$ on the vertices of the triangle,
see Figure~(\ref{1})
\end{proof}

\paragraph{}
In the case if $M$ is a subset of a $2d$-linear subspace $N$ of $E$, then the curvature is 
concentrated on the boundary of $M$. If $M$ is convex in $E$, a polyhedron given as the convex hull of finitely many points, then 
the curvature is located on the boundary. This is a situation already known to the Greeks. For example, if
$M$ is a triangle with angles $\alpha,\beta,\gamma$ embedded in $E=\mathbb{R}^2$, then the curvature $\kappa$ is concentrated
on the vertices $\{a,b,c\}$ with angles $\{ \alpha,\beta,\gamma\}$ of the triangle and  $\kappa(a)=\alpha/\pi, \kappa(b)=\beta/\pi$
and $\kappa(c)=\gamma/\pi$. Gauss-Bonnet $\int_M K \; dV(x) = 1$ rephrases that the angle sums to $\pi$. 
If $M$ is a spherical triangle (part of a sphere of radius $1$) and $dV$ is normalized to have volume $1$, then $K(x)=1$. 
This gives $\int_{{\rm int}(M)} K dV + (\pi-\alpha)/(4\pi) + (\pi-\beta)/(4\pi)+(\pi-\gamma)/(4\pi) = 1$ which 
is a Gauss-Bonnet reformulation of to Harriot's formula $\alpha+\beta+\gamma= \pi + {\rm Area}(M) (4\pi/{\rm Area}({\rm Sphere}))$.
Note we have used as always here a probability measure $dV$ on the sphere. We see this formula as an index expectation
formula by embedding the spherical triangle (a 2-manifold $M$ with boundary) in $E=\mathbb{R}^3$. The probability space
of linear functions is represented by unit vectors in $\mathbb{R}^3$. 

\paragraph{}
The curvature function $K(x)$ constructed as such is still {\bf local}. This means that we could use an 
arbitrary small neighborhood $U$ of $x$ in $M$ and would still get the same value. 
If we cut the manifold and just look at balls $B_r(x)$, then  
the index expectation curvature still exists and gets concentrated more and more on the boundary 
when $r \to 0$. It produces a Dirac point measure of weight $1$ at $x$ in the limit. Now,
taking a point away from a manifold has the effect of subtracting $1-\chi(S_r(x))$ changing as such
the Euler characteristic of $M$. This means that if $M$ is odd-dimensional, removing a point
adds one to the Euler characteristic and if $M$ is even dimensional, subtracts one to the Euler 
characteristic. This is obvious also from the fact that removing a point changes the
Betti number $b_{{\rm dim{M}}-1}$ by increasing it by $1$. The fact appears here  with Gauss-Bonnet. 

\paragraph{}
Let us formulate a special case which will be useful later and which goes back to 
Banchoff \cite{Banchoff1967}. Let us call $M$ a {\bf convex polyhedron} if 
$M$ is a Riemannian $2d$-manifold with piecewise smooth boundary embedded in an Euclidean
space $E$ such that all faces are linear $2d-1$ dimensional planes in $E$
and can be written as the convex hull of its corners $\{ v_1, \dots, v_n \}$.  
We do not assume that the curvature of $M$ is zero. 

\begin{lemma}
If $M$ is a convex polyhedron in the above sense, the index expectation 
is supported on the interior producing a function $K(x)$ and on the corners
producing a point measure $dk$ there. Gauss-Bonnet is  
$$  \chi(M) = \int_{{\rm int}(M)} K(x) \; dV(x) + \sum_{k=1}^n \kappa_k \; . $$
The function $K(x)$ in the interior is the normalized Gauss-Bonnet-Chern integrand. 
\end{lemma}
\begin{proof}
Given a linear function $f=f_a$ in $E$. It induces a linear function on each face
which has for for almost all cases of $a$ no critical point away from the corners. 
\end{proof} 

\paragraph{}
The second example will show that if $M$ has curvature sign $e$ in the interior of $M$,
then we can construct an index expectation curvature function $K(x)$ which has sign $e^d$
in the interior and which has the same $\kappa_k$ curvature measure on the boundary.
The weights $\kappa_k$ are the same than in the Gauss-Bonnet-Chern
reformulation above. If we take a compact Riemannian manifold $M$ and chop it up into piecewise
polyhedra, construct the functions $K(x)$ in each part and noticing that in 
the curvatures at corners disappear when gluing the polyhedra together, 
we expect to get to the Conjecture~(\ref{revival}).

\section{Morse functions}

\paragraph{}
Morse functions can in general be constructed by a sufficiently rich parametrized set of functions
on $M$. If the dimension of the parameter space $\Lambda$ is large enough, then by Sard's theorem,
almost all functions in $\Lambda$ are Morse. Various things come naturally: distance functions,
heat kernel functions or height functions coming from embeddings. The simplest case is using a
height function after an embedding: for the construction of Morse functions on $M$ we can embed  $M$ 
in an Euclidean space $E$ and take the set of linear functions
$f(x)=v \cdot x$ with $v$ in the unit sphere $S$ of $E$ induces a probability space
$(\Omega,\mathcal{A},\mu)$ if $\mu$ comes from the normalized volume probability measure on $S$.

\paragraph{}
Taking linear functions in an ambient linear space is a common situation and the following result
has already been known to Morse:

\begin{lemma} 
If $M$ is embedded in $E$, then $f(x) = a \cdot x$ restricted to $M$ is Morse for almost all $a$. 
\end{lemma}
\begin{proof}
The gradient $g=df$ defines a smooth map $g: M \to E$. Almost all values $a$ are 
regular values for $g$ by the Sard theorem. That is, if $g(x)=a$ is a regular value, then $dg(x)$ has 
maximal rank. But $g(x)=a$ means $df=0$ and $dg(x)=H(x)$ is the Hessian which has maximal rank. 
\end{proof} 

\paragraph{}
Given $d$ Morse functions $f_1,\dots,f_d$ on $M$, we can lift them locally to functions on $O(M)$ by
setting $f_k(x,t) = f_k(x)$. While the critical points of $x \to f_k(x,t)$ do 
not depend on the coordinate system in $O(M)$ having equipped a neighborhood with a basis will
allow us to produce more complicated Morse functions. 
Given a point $(x,t)$ in $O(M)$, we have a local coordinate system at $x$, allowing to write
the point $x \in M$ as $(z_1, \dots ,z_d)$ with $z_k=(x_{2k-1},x_{2k}) 
\in \mathbb{R}^2$. We will now construct a new Morse function $f$ from a list 
$f_k$ of Morse functions, where the choice of the coordinate system will matter. 

\paragraph{}
The Morse functions we are going to need require a bit more regularity. Let $t$ be a fixed
coordinate frame on $M$. Let us fix a probability space of $\Omega$ of Morse functions on $M$. 

\begin{lemma}
For almost all $w_1, \dots, w_d \in M^d$ and
almost all $f \in \Omega$,  all the functions
$z_k \to f(w_{k,1},\dots$, $w_{k,k-1},z_k,w_{k,k+1}$, $\dots,w_{k,d})$ are Morse.
\end{lemma}

\begin{proof}
We proceed as usual with Sard theorem (see e.g. \cite{NicoalescuMorse}). 
The probability space $\Omega$ is in our
case the parameter manifold $\Lambda=\Omega \times M^d$.
For each $\lambda \in \Lambda$ we get a function $g_{\lambda}(z)$. It is a general
fact that in this case for almost all $\lambda \in \Lambda$ we have a Morse function. 
\end{proof}

\section{Product manifold}

\paragraph{}
The product situation is a case, where the Gauss-Bonnet-Chern integrand $K(x)$ 
has the right sign everywhere. And that despite the fact that many sectional curvatures 
are zero. (A product of positive curvature manifolds has only non-negative curvature and is
never a positive curvature manifold with the standard product metric.)
In this example already, how we can interpret the result integral geometrically and
illustrate the main idea on how to build from sectional curvature functions $K_k(x)$
a new global curvature function $K(x)$ which has the right sign. In the warped case,
where the metric on $M$ can depend on all coordinate parts $M_k$, the curvature $K(x)$
already is non-local, but this would already require us to cut the manifold up. 

\paragraph{}
Let us look at the case when the $2d$-manifold is a product manifold
$M=M_1 \times \cdots \times M_d$ of 
$2$-manifolds $M_k$ and the metric on $M_k$ has the same sign $e$ for all $k$. This is a case, 
where one can get a curvature $K$ geometrically by taking probability spaces $\Omega_k$ of
Morse functions on $M$ and take $\Omega=\Omega_1 \times \cdots \times \Omega_d$.
Let us assume first that the metric on $M$ is the product metric defined by Riemannian
metrics $g_k$ on $M_k$. Define the function
$$  f(x)=f_1(x_1) + \cdots + f_d(x_d)  \; . $$ 
It has the property that the index of a critical point $x$ of $f$ is the product of
the indices of the corresponding critical points $x_k$ of $f_k$ so that
$$ K(x) = E[i_f(x)] = E[\prod_k i_{f_k}(x_k)] = \prod_k E[i_{f_k}(x_k)] = \prod_k K_k(x)$$ 
by independence (which implies decorrelation) showing that we can produce a curvature 
function $K$ on the product. This worked, even so not all sectional curvatures of $M$ 
have the sign $e$. All mixed planes spanned by vectors $v,w$ with $v$ in $M_k$ and $w$ in $M_l$ 
with $k \neq l$ have zero curvature. 

\paragraph{}
The classical Gauss-Bonnet result for $2$-manifold shows, again
by taking expectations of the product and noticing that the individual
pieces are still independent also with respect to the space
variables $x_k$, that $\chi(M)=\prod_{k=1}^d \chi(M_k)$. This product formula
is a well known property of the Euler characteristic functional. It follows
for example from Euler-Poincar\'e $\chi(G)=p(-1)$ which expresses $\chi(M)$ using
the Poincar\'e polynomial $p_M(t) = \sum_{k=0}^{2d} b_k t^k$ or then the K\"unneth formulas 
$p_{M \times N}(t) = p_M(t) p_N(t)$ or by triangulating $M_i$ with $m_i$ triangles 
and building up the product space as a CW-complex with $\prod_i m_i$ cells. 
Or then also by Poincar\'e-Hopf and seeing that for the single function 
$f(x)=f_1(x_1) + f_2(x_2) + \cdots + f_d(x_d)$ has critical points whose indices 
are the product of the indices of the individual functions.

\paragraph{}
Here is a first statement for illustration. 

\begin{propo}
On the product $2d$ manifold $M = M_1 \times \cdots \times M_d$, where each $M_k$ is a compact 
Riemannian $2$-manifold with constant curvature sign $e_k$ and the metric on $M$ is the product metric,
there is an index expectation curvature function $K(x)$ which has sign $\prod_k e_k$ everywhere. 
The curvature function $K(x)$ is local. 
\end{propo}

\begin{proof}
Let $(\Omega_k,\mu_k)$ be the probability space on Morse functions of $M_k$ which produces
the Gauss curvature on $M_k$. This is the case when $M$ is embedded in an ambient Euclidean space $E$
and then the probability space consists of all linear functions $f_a(x) = a \cdot x$ with $|a|=1$. 
Given functions $f_k \in \Omega_k$, define the new function 
$$ f(x_1, \cdots, x_d) = \sum_{k=1}^d f_k(x_k)  \; . $$
Its critical points are simultaneously critical points of $f_k$ so that $f$ is automatically Morse
on $M=M_1 \times \cdots \times M_d$. Furthermore, the indices satisfy 
$$ i_f(x) = \prod_{k=1}^d i_f(x_k)  \; . 	$$
\end{proof}

\paragraph{}
Since Euler characteristic is independent of the metric, the Euler characteristic of $M$ 
is still $\prod_k \chi(M_k)$. There is nothing new yet with respect to the Hopf conjecture. 
The, the product manifold case is historically interesting. It motivated
the question to look for a positive curvature metric on manifolds like $S^2 \times S^2$.
The product case appeared as an illustration in a talk of Hopf \cite{Hopf1953}
given in Italy.

\paragraph{}
Here is the proof of the same statement without using any probability:

\begin{proof}
We just could have directly defined
$K(x) = \prod_k K_k(x_k)$, where $K_k$ are the usual Gaussian curvatures
on $M_k$ normalized so that they define	probability measures on $M_k$ giving $\int_{M_k} K_k(x) dV(x)=\chi(M_k)$ 
and where $x=(x_1,\dots, x_d)$ is a point in the product manifold $M$. Now, when integrating
$$ \int_M K(x) dV(x) = \int_{M_1} \cdots \int_{M_d} K_1(x_1) \cdots K_d(x_d) dV(x_1) \cdots dV(x_d) \; ,  $$
this is $\prod_k \int_{M_k} K_k(x_k) \; dV(x_k) = \prod_k \chi(M_k)$ which has sign $\prod_k e_k$ as $\chi(M_k)$
as the sign $e_k$ of $K_k$.
\end{proof} 

\paragraph{}
What happens in the {\bf warped product case},
where we start with a metric $g_k$ on $M_k$ can depend on the other parts? 
Changing the metric as such does not change the Euler characteristic of $M_k$ 
nor of $M$. We have to use 
an extended probability space $\prod_{k=1}^{d} (\Omega_k \times M_k)$ 
with the product measure $\prod_{k=1}^d \mu_k \times dV_k$.
The function is already what we will use later on in general:
$$ f(x,y_1,\dots,y_d) = \sum_{k=1}^d  f_k(y_1, \dots, y_{k-1},x_k,y_{k+1}, \dots, y_d) \;. $$
Now, this expression already illustrates that some correlations can occur which will require
to split the manifold up. 

\section{Manifolds with boundary: statement}

\paragraph{}
We prove now a statement for compact $2d$-manifolds $M$ with boundary $\delta M$ which are 
part of a parallelizable non-compact manifold $N$ of the same dimension which is itself part of 
a larger dimensional Euclidean space. The construction works for a general
metric on $N$, but we assume that a definite sign $e$ of sectional curvatures is present on $N$.
Without loss of generality we can assume that $N=\mathbb{R}^{2d}$ is equipped with a fixed metric
and that we restrict $M$ to be in a small ball of $N$> 

\paragraph{}
As the compact $2d$-manifold $M$ now has a boundary, we look at probability spaces 
$(\Omega,\mathcal{A},\mu)$ of Morse functions $h$ which vanish to all orders at the boundary of $M$ and which satisfy 
the Morse condition at each critical point in the interior of $M$. For a fixed positive 
parameter $\epsilon>0$ we can look at a function $\rho$ which satisfies 
$\rho(x)=1$ for $x$ in distance $d \geq \epsilon$ from the boundary, zero at the boundary
to all orders. For almost all linear functions $f=a \cdot x$ in an ambient
space $g(x) = \rho(x) f(x)$ is Morse.

\paragraph{}
We also always assume that that the probability space has enough symmetry. This assures that 
the constructed curvature $K$ is piecewise smooth in the interior of the pieces. 
This is the case if the Morse functions come
from linear functions in an ambient Euclidean space $E$. It implies for example that 
if we get the same probability space if $M$ is rotated around a point $x$ in $M$ 
perpendicular to $M$. We also always fix a global
section $t(x)$ of the orthonormal frame bundle $O(N)$. This means that at every point of $M$
we have an orthonormal basis. This can be obtained by choosing coordinate system at one 
point $x_0 \in N$ then transport it to other places using translations in $N$. 

\paragraph{}
If $f$ is a function for which $f$ and all derivatives vanish at the boundary $\delta M$,
then the Poincar\'e-Hopf theorem $\sum_x i_f(x) = \chi(M)$ 
still works and the index expectation curvature $K(x)$ still 
satisfies Gauss-Bonnet. This is a situation we have looked at
in \cite{DiscreteHopf} already, but were we did not yet use enough variables to 
decorrelate. The fact that we can not carry over the sign to boundary has not changed: 
if $M$ carries positive curvature for example, like if $M$ is part of a half sphere 
with $p$ holes in it, then the Euler characteristic is $1-p$ so that each of the holes
making up the boundary has to carry a curvature $-1$. Instead of having the functions
going to zero at the boundary we look instead at the situation of a {\bf manifold
with boundary} and allow curvature $\kappa$ on the boundary (which can there be quite
an arbitrary measure at first). 

\begin{propo}
If $M$ is part of a parallelizable open manifold $N$ with curvature signature $e$ 
and $M$ has a boundary $\delta M$, there is an absolutely continuous 
index expectation curvature $K(x)$ such that for a fixed $x$ in the interior of $M$,
then $K(x)$ has the sign $e^d$. There is an additional measure $d\kappa$ on the boundary
such that $\int_M K(x) \; dV(x) + \int_{\delta M} d\kappa(x) = \chi(M)$. 
\end{propo}

\section{Constructing of the probability space} 

\paragraph{}
As the ambient space $N$ is parallelizable, we can 
use a global section $t(x)$ of the orthonormal frame bundle $O(U)$ to get an orthonormal 
basis $t(x)$ at every point $x \in N$ and so every point in $M$. This basis allows 
to define global coordinates $z=(z_1,\dots,z_d)$ in $M$, where $z_k = (u_k,v_k)$ 
are the coordinates of the $d$ orthonormal $2$-plane. One can think of the $z_k$ as complex
numbers, but we will not look at any functions compatible with a complex structure but
have smooth functions on real manifolds. We keep a particular choice of a smooth section 
$t(x)$ of $O(U)$.

\paragraph{}
{\bf Definition.} This is the main construction of combined Morse functions: 
given $d$ Morse functions $f_k$ in $(\Omega,\mathcal{A},\mu)$, and $d$ vectors
$w_{j}=(w_{j,1},...,w_{j,d}) \in M$, define the function 
$$ f(z,w_1,w_2, \dots, w_d) 
     = \sum_{k=1}^d f_k(w_{k,1}, \dots, w_{k,k-1},z_k,w_{k,k+1}, \dots, w_{k,d})  \; . $$
The construction of $f$ depends on the frame which was chosen. If we change the coordinate
system, then the function $f$ also changes. Much more dramatically than for a single function 
alone because also the critical points change. 
If $w_{1},\dots, w_{d},f_1, \dots ,f_d$ are fixed in the generic set, where the 
$z_k \to f_k(w_{k,1},\dots,z_k,\dots,w_{k,d})$ are Morse, then this is a Morse function in 
$z$ because of the following product formula:

\begin{lemma}
$$  i_f(x; w_1, \dots, w_d) = \prod_{k=1}^d i_{f_k}(z_k; w_1, \dots, w_d)  \; . $$	
\end{lemma}

\begin{proof}
Fix the $d$ points $w_1, \dots, w_d$ in $M$. 
If $x$ is a critical point, then each $z_k$ is a critical point of $f_k$. 
The Hessian of $f$ is block diagonal. The index $i_f(x) = (-1)^{m(x)}$, where
$m(x)$ is the number of negative eigenvalues of the Hessian $H(x)$. 
As $i_{f_k}(x) = (-1)^{m(z_k)}$ and $m(x) = \sum_k m(z_k)$, the statement follows.  
\end{proof} 

\paragraph{}
{\bf Example.} Here is the case $d=3$, where we deal with a $6$-manifold $M$. 
Given $f_1,f_2,f_3 \in \Omega$ and $(w_{k,1},w_{k,2},w_{k,3})$ are in $M$ for $1 \leq k \leq 3$,
the function $f$ is
$$ f(z_1,z_2,z_3) = f_1(z_1,w_{1,2},w_{1,3}) + f_2(w_{2,1},z_2,w_{2,3}) + f_3(w_{3,1},w_{3,2},z_3)\; . $$
This is very close to a product situation with $2$-manifolds
Allowing to play with the probability space was what allowed the decorrelation of the 
section curvatures.  Here are three remarks:

\paragraph{}
{\bf 1)} This is essentially a product situation. However, we have introduced $d$ dummy points
$w_k = (w_{k,1}, \dots ,w_{k,d})$  $\in M$,  in order to keep the individual parts independent. 
The introduction of the variables $w_{k} \in M$ renders the curvature non-local as
Every part of $M$ can contribute to $f(x)$. This leads to the next remark, which we did not
consider yet in \cite{DiscreteHopf} (where we still used one one additional parameter $w$ and did 
not yet get full independence). The result stated there in a ball $B_r$ with $r$
smaller than the injectivity radius needs to be replaced by $r$ small enough. The value of $r$
which work depends on global bounds of derivatives of the metric near the point under
consideration. 

\paragraph{}
{\bf 2)} Only for $w_{k,j}=z_j$ with $j = 1,\dots,d$ do we have the situation that the parametrized surfaces 
$z_k \to f(w_{k,1}, \dots ,z_k, \dots ,w_{k,d})$ are built from geodesic paths starting at $x$. If 
$M$ is small, then the sign of the curvature of $z_k \to f(w_{k,1}, \dots ,z_k, \dots ,w_{k,d})$ is
the same than the sign of the curvature $\tilde{z}_k \to f(w_{k,1}, \dots ,\tilde{z_k}, \dots ,w_{k,d})$
if $\tilde{z}$ parametrize a geodesic spay $\exp_{(w_{k,1},\dots ,z_k,\dots w_{k,d})}(D)$ centered at
$(w_1,\dots,z_k,\dots,w_d)$. The reason for the concern is that the exponential maps 
at different points do not commute in Riemannian geometry: the relation
$\exp_x(u) \exp_y(v) = \exp_y(v) \exp_x(u)$ does not hold in general. 

\paragraph{}
{\bf 3)} We can look at $i(x)$ and $i_k(x)$ as point measures or distribution-valued random variables over the probability 
spaces $(\Omega \times M)^d$ or $(\Omega \times M)$. When we write $K(x)=E[i(x)]$ we mean the expectation of this. This
means that for any smooth test function $\phi$, the expectation of $i(\phi)$ is $K(\phi)$. 
In all the cases we look at the function $K$ is smooth in the interior of $M$. 
When looking at the product formula above, one can be concerned with the fact that distributions can not be 
multiplied in general. But we see by looking at the situation that is not of concern here as the individual
distributions apply to different dimensions. It is like writing the Dirac measure 
$\delta_{\{x_0,y_0\}}(x,y)$ in $\mathbb{R}^2$ as a product of Dirac measures 
$\delta_{\{ x_0\}}(x) \delta_{\{ y_0\}}(y)$. 

\section{Constructing the curvature}

\paragraph{}
Again assume that $M$ is a compact $2d$-manifold with boundary which is part of 
$N=\mathbb{R}^{2d}$. We can use a global coordinate system of $N$. 
Given a probability space of functions $f_k$ as done above, we define the probability space $(\Omega \times M)^d$
equipped with the product probability measure $(\mu \times dV^d)^d$, where $\mu$ is the assumed measure on Morse
functions $\Omega$ and $dV$ is the normalized volume measure on $M$. 
Denote with $E[X]$ the expectation with respect to this probability space and with $E_k[X]$ the 
expectation with respect to one of the factors $(\Omega \times M^d)$. 

\paragraph{}
{\bf Definition.} Define the {\bf index expectation curvatures} 
$$  K(x) = E[i(x)] $$ 
and 
$$  K_k(x) = E[i_k(x)] \; , $$
where $i_k(x),i(x)$ are the index functions as defined above which are 
distribution valued-random variables. \\

If the ambient space $N$ has a definite curvature sign $e$ and $(\Omega,\mu)$ is the probability space
on Morse function which produces the Gauss curvature, then the $K_k$ have the
sign $e$ and the curvature function $K$ has the sign $e^d$: 

\begin{lemma}[Product lemma]
In the interior of $M$ we have $K(x)= \prod_{k=1}^d K_k(x)$. 
\end{lemma}

\paragraph{}
\begin{proof}
If two random variables $X,Y$ are independent, then they are 
decorrelation, meaning $E[X Y] = E[X] E[Y]$. In our case, we take expectations over point measures
which can be seen as generalized functions, linear combinations of Dirac measures. The $i_k(x)$ and $i_k$
are $C(M)^*$ valued random variables, where $C(M)$ is the space of continuous functions on $M$  and $*$ denotes
the dual space. To detail this out we see $i(x)$ and $i_k(x)$ as point measures and see the identity 
as an identity for distributions. It becomes an identity for traditional random variables when
using test functions. For fixed $x$ in the interior, we can use this for the random variables 
$X_i = i(\phi)$ with a suitably narrow test function $\phi$. It is also enough for every open set $U$ in $M$
look at the $\int_U i_k(x) dV$ which is $X_{k,U} = \sum_{x \in U} i_k(x)$, a finite sum. These are now
integer valued traditional random variables on $(\Omega \times M)$. Now, $X_{k,U}$ and $X_{l,U}$ are 
independent if $k \neq l$. 
\end{proof}

\paragraph{}
Here again, the restriction of $M$ having to be cut small comes in and this is an 
other obstacle for making things global without cutting up the manifold. We are using
the parametrized surfaces 
$$  \Sigma: z_k \to f(w_{k,1},\dots ,z_k, \dots ,w_{k,d})  $$
which is different than looking at geodesic sprays 
$$ \Delta: \tilde{z}_k \to f(w_{k,1}, \dots ,\tilde{z_k}, \dots ,w_{k,d})  \; , $$
where $\tilde{z}_k = (u_k,v_k)$ parametrize the surface spread out by 
geodesics starting at $(w_{k,1}, \dots$ , $z_k$, $\dots ,w_{k,d})$. These two surfaces go through
the same point in $M$ and are tangent, but they are not the same. If $N$ is fixed
and we chose a small enough $M$, then the curvatures of all these surfaces
have the sign $e$. For $\Delta$, we can get some correlations between the indices.
For $\Sigma$ we have independence but we do not have surfaces which 
consist of geodesics. What happens if $M$ is small enough so that the curvature 
of the first has the same sign $e$ than the curvature of the second which 
is assumed to be $e$. 

\section{Illustrations}

\begin{figure}[!htpb]
\scalebox{0.6}{\includegraphics{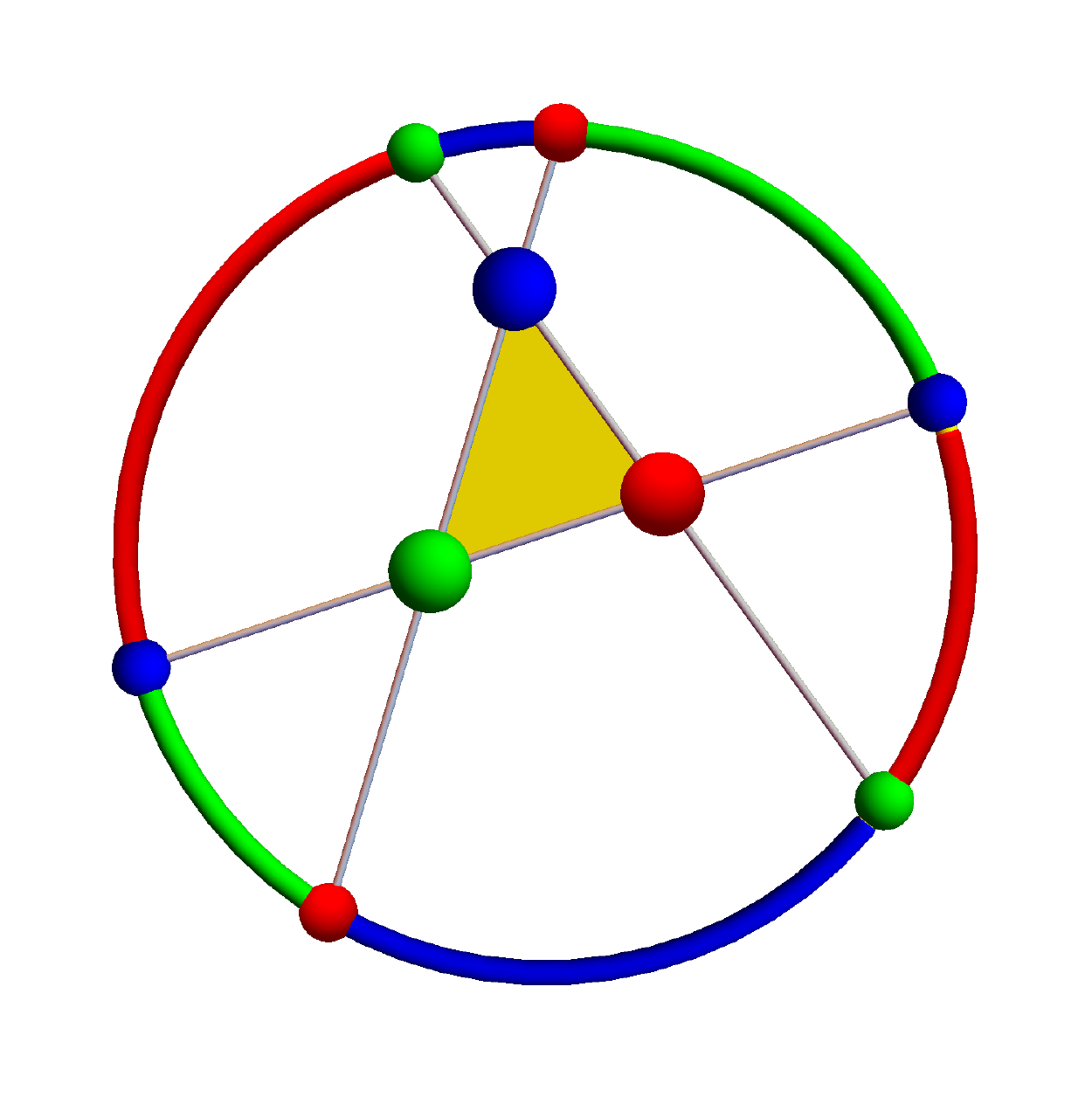}}
\label{1}
\caption{
The oldest Gauss-Bonnet result is the school geometry rule $\alpha+\beta+\gamma=\pi$
for flat triangles $M$. We can see it integral theoretically: we look at $M$ as
a 2-manifold with piecewise smooth boundary embedded in $N=E=\mathbb{R}^2$. The probability 
space is the unit circle in $E$ with normalized measure $dV$. In this
case, the critical points are all minima as then $i_f(x) = \{ y \in M | 1-\chi(S_r(y) \cap M)  \neq 0 \}$
is equal to $1$. The curvature is located on the three vertices and given by $\alpha/\pi$
and $\beta/\pi$ and $\gamma/\pi$, where each of the three probabilities are the combined length
of two arcs (normalized so that the entire circle $S=\Omega \subset E$ is the probability space). 
The same curvatures would result when embedding $M$ into a larger dimensional Euclidean space $E$.  
}
\end{figure}

\begin{figure}[!htpb]
\scalebox{0.6}{\includegraphics{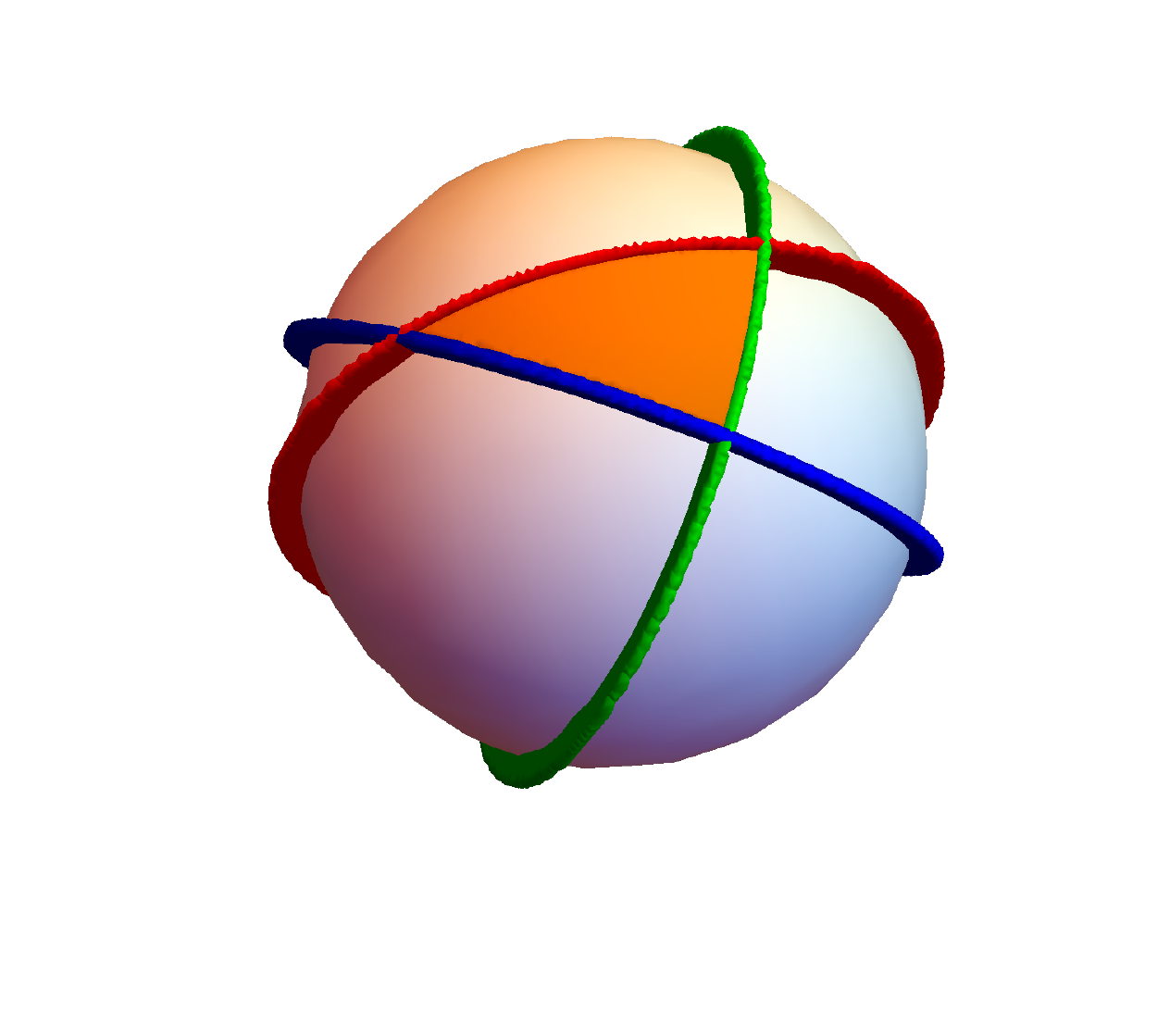}}
\label{2}
\caption{
A related example is when
$M$ is a spherical triangle embedded in $E=\mathbb{R}^3$, then there is some
curvature of sign $e=+1$ in the interior of $M$. It is the normalized area of the triangle itself
and its dual part. Then there is mass on the vertices. These are each given by 
two slices in the partition defined by the grand circles of the triangle. Translating
the probabilities gives the Harriot formula from the 15'th century. The sum of the 
angles in a spherical triangle is larger than $180$ degrees and the excess is given 
by the area of the spherical triangle. 
}
\end{figure}

\begin{figure}[!htpb]
\scalebox{0.6}{\includegraphics{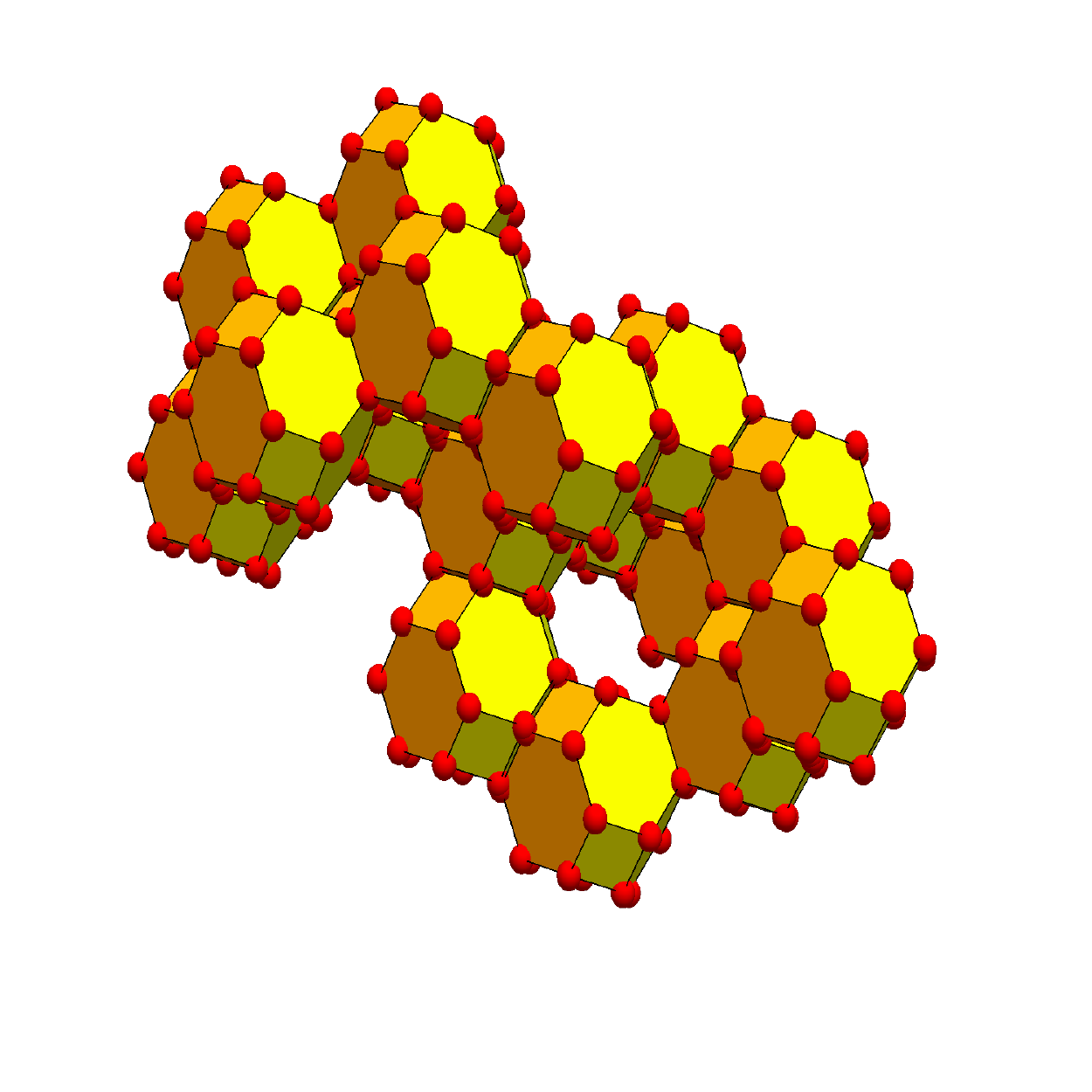}}
\label{2}
\caption{
The story told here shows that it is not possible to find a global smooth curvature function
$K(x) = \prod_k K_k(x)$ in general; we have therefore to cut the manifold up into smaller polyhedra
(for example simplices after a triangulation). In each piece $U$ we can find a curvature function $K_U$ and some 
discrete mass on the boundary. 
When looking at index expectation on such a manifold $U$ with piecewise-linear 
boundary, there is a curvature $K_U(x)$ in the interior and measure
$\kappa_U$ supported on the vertices of $U$. When gluing different such pieces $U$
together, then the curvature in the interior disappears. This picture leads to the
revival of the algebraic Hopf conjecture. 
}
\end{figure}

\bibliographystyle{plain}

\end{document}